\documentclass[a4paper]{amsart}
\usepackage{graphicx}
\usepackage[colorlinks, linkcolor= blue]{hyperref}
\usepackage{pdfsync}

\usepackage[T1]{fontenc}
\usepackage[utf8]{inputenc}

\usepackage{amsmath, amssymb, amsfonts, amscd, amsthm}

\newcommand{\myexp}{\,\textnormal{\sc exp}\,}
\newcommand{\C}{\mathbb{C}}
\newcommand{\A}{\mathcal{A}}

\newcommand{\z}{\mathbb{Z}}

\renewcommand{\a}{\mathbf{A}}

\newcommand{\f}{\mathbb{F}}
\newcommand{\g}{\mathbf{G}}

\renewcommand{\O}{\mathcal{O}}
\newcommand{\h}{\mathcal{H}}
\renewcommand{\k}{\mathcal{K}}
\renewcommand{\b}{\mathcal{B}}
\newcommand{\Int}{\operatorname{Int}}
\newcommand{\Inn}{\operatorname{Inn}}

\newcommand{\mm}{\times}
\newcommand{\ad}{{}}

\newtheoremstyle{pedro}{}{}{\itshape}{}{\sc}{~--}{ }{\thmname{#1}\thmnumber{ #2}\thmnote{ (#3)}}

\newtheoremstyle{pedrodef}{}{}{}{}{\sc}{~--}{ }{\thmname{#1}\thmnumber{ #2}\thmnote{ (#3)}}

\theoremstyle{pedro}
\newtheorem{lem}{Lemma}[section]

\newtheorem{thm}[lem]{Theorem}

\newtheorem{prop}[lem]{Proposition}

\newtheorem{coro}[lem]{Corollary}

\theoremstyle{remark}

\newtheorem{rmk}[lem]{Remark}

\theoremstyle{pedrodef}

\newtheorem{ex}[lem]{Example}

\title{Examples of quantum algebra in positive characteristic}

\author{Pierre Guillot}

\address{
Universit\'{e} de Strasbourg \& CNRS\\
Institut de Recherche Math\'{e}matique Avanc\'{e}e\\
7~Rue Ren\'{e} Descartes\\
67084 Strasbourg, France}

\email{guillot@math.unistra.fr}

\numberwithin{equation}{section}

\begin{document}

\maketitle

\begin{abstract}

There have been few examples of computations of Sweedler cohomology,
or its generalization in low degrees known as lazy cohomology, for
Hopf algebras of positive characteristic. In this paper we first
provide a detailed calculation of the Sweedler cohomology of the
algebra of functions on$\left(\z/2 \right)^r$, in all degrees, over a
field of characteristic~$2$. Here the result is strikingly different
from the characteristic zero analog.

Then we show that there is a variant in characteristic~$p$ of the
result obtained by Kassel and the author in characteristic zero, which
provides a near-complete calculation of the second lazy cohomology
group in the case of function algebras over a finite group; in
positive characteristic the statement is, rather surprisingly,
simpler.

{\tiny {\bfseries Keywords:} Hopf algebras; Sweedler cohomology; lazy cohomology; Drinfeld twists; $R$-matrices.   }

\end{abstract}


\section{Introduction}

Sweedler cohomology was defined in~\cite{sweedler}. Given a field~$k$,
a {\em cocommutative} Hopf algebra~$\h$, and an~$\h$-module
algebra~$A$, Sweedler defines cohomology groups which we will
write~$H^n_{sw}(\h, A)$, for~$n \ge 1$. In fact we shall only discuss
the case where~$A= k$, the base field, viewed as an~$\h$-module via
the augmentation, and write simply~$H^n_{sw}(\h)$. 

When~$\h = k[G]$, the group algebra of the finite group~$G$, one
has~$H^n_{sw}(\h) = H^n(G, k^\times)$; and if~$\h = U(\mathfrak{g})$,
the universal enveloping algebra of the Lie algebra~$\mathfrak{g}$,
one has~$H^n_{sw}(U(\mathfrak{g})) = H^n(\mathfrak{g}, k)$. The first
virtue of Sweedler cohomology is thus to unify these two classical
cohomology theories. 

There are few other examples of Hopf algebras for which the Sweedler
cohomology is known. An easy way to construct Hopf algebras is, of
course, to consider algebras of functions on groups. In the simplest
case thus, we may take a finite group~$G$ and consider the
algebra~$\h= \O_k(G)$ of~$k$-valued functions on~$G$. However, $\h$ is
only cocommutative when~$G$ is abelian, which severely restricts the
choices. What is worse, in the familiar case when~$k$ is the field of
complex numbers~$\C$, one can use the discrete Fourier transform to
get an isomorphism~$\O_k(G) \cong k[\widehat{G}]$, where~$\widehat{G}$
is the Pontryagin dual of~$G$; we are thus reduced to the group
algebra case and will not get anything new.

However in positive characteristic, the Fourier transform is not
available, and the Hopf algebra~$\O_k(G)$ is genuinely different from a
group algebra. Our main result in this paper involves the elementary
abelian~$2$-groups, that is groups of the form~$G=
\left(\z/2\right)^r$. For these, over any field~$k$ of characteristic
zero or~$p > 2$, we have~$\O_k(G) \cong k[G]$, so~$H^n_{sw}(\O_k(G)) =
H^n(G, k^\times)$. For example for~$G = \z/2$, the latter is~$k^\times
/ (k^\times)^2$ when~$n$ is even, and~$\{ \pm 1 \}$ when~$n$ is
odd. When~$k$ has characteristic~$2$, by contrast, we obtain the
following result (Theorem~\ref{thm-main-elemab} in the text):

\begin{thm} \label{thm-intro-zedmodtwo}
Let~$k$ be a field of characteristic~$2$. The Sweedler cohomology
of~$\O_k(\left(\z/2 \right)^r)$ is given by
\[ H^n_{sw} (\O_k(\left(\z/2 \right)^r)) = \left\{ \begin{array}{l}
0 ~\textnormal{for}~ n\ge 3 ~\textnormal{or}~ n=0 \, , \\
~\\
\left(\z/2 \right)^r ~\textnormal{for}~ n=1 \, , \\
~\\
\left( k / \{ x + x^2 : x \in k \} \right)^{\oplus r}
~\textnormal{for}~ n=2 \, . 
\end{array}\right.  \]
When~$k$ is an algebraically closed field, we have in particular $H^2_{sw}(\O_k(\left(\z/2 \right)^r)) = 0$.
\end{thm}

In fact for any ring of characteristic~$2$, the cohomology groups
vanish in degrees greater than~$2$.

As far as non-cocommutative Hopf algebras go, there is at least a
definition in low-degrees of the so-called ``lazy cohomology
groups''~$H^n_\ell(\h)$, for~$n= 1, 2$, with no restriction on the
Hopf algebra~$\h$; of course these agree with Sweedler's cohomology
groups when~$\h$ happens to be cocommutative. Lazy cohomology was
defined originally by Schauenburg, and systematically explored
in~\cite{bichon}. This opens up the exploration of~$H^n_\ell(\O_k(G))$
for any finite group~$G$, mostly for~$n=2$: there is not much mystery
held in the case~$n=1$, since one has~$H^1_\ell(\O_k(G))=
\mathcal{Z}(G)$, the centre of~$G$.

 The groups~$H^2_\ell(\O_k(G))$ were investigated by Kassel and the author
 in~\cite{kassel} in characteristic~$0$. Our main result is as
 follows. Let~$\b(G)$ denote the set of all pairs~$(A, b)$ where~$A$
 is an abelian, normal subgroup of~$G$, and~$b$ is an alternating,
 non-degenerate, $G$-invariant bilinear form~$\widehat{A} \times
 \widehat{A} \to k^\times$. The point is perhaps that~$\b(G)$ is easy
 to describe in finite time. We have then constructed a map of sets
\[ \Theta \colon H^2_\ell(\O_k(G)) \longrightarrow \b(G)  \]
with good properties. In particular, when~$k$ is algebraically closed,
the fibres of~$\Theta $ are finite (and explicitly described), which
proves that~$H^2_\ell(\O_k(G))$ is finite in this situation. It is easy to
use the map~$\Theta $ to compute~$H^2_\ell(\O_k(G))$ in many cases, and
we have thus been able to show that this group can be arbitrary large,
and possibly even non-commutative. 

In this paper, we extend the results of {\em loc.\ cit}.\ to positive
characteristic. It turns out that the result is easier in this
case. The following is made precise in
Theorem~\ref{thm-main-guillotkassel} in the text. 

\begin{thm} \label{thm-intro-guillotkassel}
When~$k$ has characteristic~$p$, the map 
\[ \Theta \colon H^2_\ell(\O_k(G)) \longrightarrow \b(G)  \]
exists with the same formal properties as in characteristic~$0$,
except that the subgroups~$A$ of order divisible by~$p$ have been
excluded from the definition of~$\b(G)$.

In particular, this map is surjective when~$k$ has characteristic~$2$.
\end{thm}

\medskip

\noindent {\em Organization of the paper}. Section~\ref{sec-defs}
recalls all the relevant definitions. The strategy for the
computations of Theorem~\ref{thm-intro-zedmodtwo} is given in
Section~\ref{sec-homology}, together with a few tools from homological
algebra. The computation itself is carried out in
Section~\ref{sec-zedmodtwo} for the case~$r=1$, which is much less
technical, and in Section~\ref{sec-generic-case} for the general
case. The last Section is devoted to
Theorem~\ref{thm-intro-guillotkassel}.

\medskip

\noindent {\em Acknowledgements.} The author wishes to thank the
referee for pointing out Theorem 4.8 in~\cite{bichon}.


\section{Definitions} \label{sec-defs}

\noindent {\em Notation.} Throughout the paper, we
write~$R_\mm$ for the multiplicative group of units in the ring~$R$.

\medskip

We recall all the relevant definitions. The Appendix contains some
material on cosimplicial objects, should the reader feel the need to
review this topic.

\subsection{Sweedler cohomology}
(See~\cite{sweedler}.) Let~$\h$ be a Hopf algebra over the
field~$k$. For each integer~$n\ge 1$, we form the
coalgebra~$\h^{\otimes n}$ and define faces and degeneracies by the
following formulae:
\begin{equation*}\label{face}
d_i (x_0 \otimes \cdots \otimes x_n)= 
\begin{cases}
  x_0 \otimes \cdots \otimes x_i x_{i+1} \otimes \cdots \otimes x_n &
  \text{for $i < n$} \, ,\\
  x_0 \otimes \cdots \otimes x_{n-1} \varepsilon (x_n)  &
  \text{for $i=n$} \, ,

\end{cases}
\end{equation*}
and 
\begin{equation*}
s_i (x_0 \otimes \cdots \otimes x_n) = x_0 \otimes \cdots \otimes x_i \otimes 1
\otimes x_{i+1} \otimes \cdots \otimes x_n \, . 
\end{equation*}

We are thus in the presence of a simplicial coalgebra. The
monoid~$Hom(\h^{\otimes n}, k)$, equipped with the ``convolution
product'', contains the group~$R^n(\h) = Reg(\h^{\otimes n}, k)$
(comprised of all the invertible elements in~$Hom(\h^{\otimes n},
k)$). Since~$Reg(-, k)$ is a functor, we obtain a cosimplicial
group~$R^*(\h)$ (sometimes written~$R^*$ for short in what follows). 

Whenever~$\h$ is cocommutative, $R^*$ is a cosimplicial abelian group.
Thus it gives rise to a cochain complex~$(R^*, d)$ whose differential
is
\[ d = \sum_{i=0}^{n+1} (-1)^i d^i  \]
in additive notation, or (as we shall also encounter it) 
\[ d = \prod_{i=0}^{n+1} (d^i)^{(-1)^i} = d^0 (d^1)^{-1} d^2
(d^3)^{-1} \cdots   \]
in multiplicative notation.

The cohomology~$H^*(R^*, d)$ is by definition the Sweedler cohomology of
the cocommutative Hopf algebra~$\h$, denoted by~$H^*_{sw}(\h)$.

\subsection{Twist cohomology \& Finite dimensional algebras} \label{subsec-defs-twist}

Now suppose that~$\h$ is a finite-dimensional Hopf algebra. Then its
dual~$\k=\h^*$ is again a Hopf algebra. In this situation, the
cosimplicial group associated to~$\h$ by Sweedler's method may be
described purely in terms of~$\k$, and is sometimes easier to
understand when we do so. 

In fact, let us start with any Hopf algebra~$\k$ at all. We may
construct a cosimplicial group directly as follows. Let~$A^n_\mm(\k) =
(\k^{\otimes n})_\times$ and let the cofaces and codegeneracies be
defined by
\begin{equation*}\label{face}
d^i = 
\begin{cases}
1 \otimes id^{\otimes n} & \text{for $i=0$} \, ,\\
id^{\otimes (i-1)} \otimes \Delta \otimes id^{\otimes (n-i)} & 
\text{for $i=1, \ldots, n-1$} \, ,\\
id^{\otimes n} \otimes 1 & \text{for $i=n$} \, ,
\end{cases}
\end{equation*}
and 
\begin{equation*}
s^i =
\begin{cases}
\varepsilon \otimes id^{\otimes (n-1)} & \text{for $i=0$} \, ,\\
id^{\otimes (i-1)} \otimes \varepsilon \otimes id^{\otimes (n-i)} & 
\text{for $i=1, \ldots, n-1$} \, ,\\
id^{\otimes (n-1)} \otimes \varepsilon & \text{for $i=n$} \, .
\end{cases}
\end{equation*}

When~$\k$ is commutative, then~$A^*_\mm(\k)$ is a cosimplicial abelian
group, giving rise to a cochain complex~$(A^*_\mm, d)$ in the usual
way. Its cohomology~$H^*(A^*_\mm, d)$ is what we call the {\em twist
  cohomology} of~$\k$, written~$H^*_{tw}(\k)$. This terminology comes
from the fact, easily checked, that an element of~$A^2_\mm(\k) =
(\k\otimes \k)_\times$ is in the kernel of~$d$ if and only if it is a
twist in the sense of Drinfeld (see equation ($\dagger$) below).

Coming back to the case when~$\k = \h^*$ for a finite-dimensional Hopf
algebra~$\h$, it is straightforward to check that~$R^*(\h)$ can be
identified with~$A^*_\mm(\k)$ (Theorem~1.10 and its proof in~\cite{kassel}
may help). 

In this paper we are chiefly interested in computing
with~$R^*(\O_k(G))$, where~$\O_k(G)$ denotes the algebra of functions on
the finite group~$G$. By the above, this is the same
as~$A^*_\mm(k[G])$, writing~$k[G]$ for the group algebra of~$G$. It
turns out to be easier to work with the latter.

The following remarks will be useful in the sequel. As the notation
suggest, it is possible to define another cosimplicial abelian
group~$A^*(\k)$ by simply taking~$A^n(\k) = \k^{\otimes n}$, the
vector space underlying~$\k^{\otimes n}$. The cofaces and
codegeneracies are exactly the same as above (they are really maps of
algebras), although the differential of the corresponding cochain
complex is now the alternating {\em sum} of the cofaces rather than
the alternating {\em product}. Still, we denote it by~$d$. Let us
explain why the cohomology of~$(A^*, d)$ must in fact be very simple.

We consider the case~$\k= k[G]$ for a finite group~$G$, and
write~$A^*$ for~$A^*(\k)$. Note that~$(A^*_\ad, d)$ does not depend on
the group structure on~$G$, and we expect its cohomology to be
trivial. Indeed, let~$X$ denote~$G$ viewed as a pointed set only. There
is an obvious cosimplicial set which in degree~$n$ is~$X^n$ (cartesian
product of~$n$ copies of~$X$), and such that the cosimplicial
group~$A^*_\ad$ is obtained by applying the functor ``free~$k$-vector
space'' to~$X^*$. In this sort of situation we may apply the following
Lemma, which we prove in the Appendix.

\begin{lem} \label{lem-goodwillie}
  Let~$X^*$ be any cosimplicial set, let~$k$ be any ring, and
  let~$k[X]^*$ be the cosimplicial~$k$-module obtained by taking~$k[X]^n$ to
  be the free $k$-module on~$X^n$. Then~$H^n(k[X]^*, d) = 0$ for~$n >
  0$.

\end{lem}

In particular, the cohomology of~$A^*$ indeed vanishes (even in
degree~$0$, in this case).

\subsection{Lazy cohomology}

When~$\h$ is not cocommutative, Sweedler's cohomology is not
defined. However, there is a general definition of low-dimensional
groups~$H^i_\ell(\h)$ for~$i= 1, 2$, called the {\em lazy} cohomology
groups of~$\h$, for any Hopf algebra~$\h$: this definition is
originally due to Schauenburg and is systematically explored
in~\cite{bichon}. Of course when~$\h$ happens to be cocommutative,
then~$H^i_\ell(\h) = H^i_{sw}(\h)$. This is perfectly analogous to the
construction of the non-abelian~$H^1$ in Galois cohomology -- note
that~$H^2_\ell(\h)$ may be non-commutative (cf~\cite{kassel}).

When~$\h$ is finite-dimensional, there is again a description
of~$H^i_\ell(\h)$ in terms of the dual Hopf algebra~$\k$. Since this
is the case of interest for us, we only give the details of the
definition in this particular situation (using results
from~\cite{kassel}, \S1). Quite simply, $H^1_\ell(\h)$ is the
(multiplicative) group of central group-like elements in~$\k$. The
group~$H^2_\ell(\h)$ is defined as a quotient. Consider first the
group~$Z^2$ of all invertible elements~$F\in\k \otimes \k$ satisfying 
\[ \Delta (a) F = F \Delta (a)  \]
(here~$\Delta $ is the diagonal of~$\k$ -- one says that~$F$ is
invariant), and 
\[ (F\otimes 1) (\Delta \otimes id)(F) = (1 \otimes F) (id \otimes
\Delta ) (F) \tag{$\dagger$} \]
(which says that~$F$ is a Drinfeld twist). The group~$Z^2$ contains
the group~$B^2$ of so-called trivial twists, that is elements of the
form~$F = (a\otimes a) \Delta (a^{-1})$ for~$a$ central
in~$\k$. Then~$H^2_\ell(\h) = Z^2 / B^2$.



\section{Homological preliminaries} \label{sec-homology}

This section prepares the ground for the next two.

\subsection{Strategy} \label{subsec-strategy}

It is classical that the cohomology of any abelian cosimplicial
group~$A^*$ can be computed by restricting attention to the
``normalized'' cocycles, that is those cocycles for which all the
codegeneracies~$s^*$ vanish. In this paper we shall focus our
attention on the last codegeneracy ($s^{n-1}$ on~$A^n$). As we shall
explain at length in this section, the cocycles for which~$s^{n-1}$
vanish can also be used to compute the cohomology.

The reason for paying special attention to this map is the following.
Put~$A^n= k[G^n]$, where~$G$ is a finite group. The cosimplicial
group we are interested in is~$A_\mm^n= k[G^n]_\mm$ as
in~\S\ref{subsec-defs-twist}. Thus~$A_\mm^n = R[G]_\mm$ for~$R=
k[G^{n-1}]$, and the map~$s^{n-1} \colon R[G]_\mm \to R_\mm$ is simply the
augmentation. So the ``unit spheres'' $S(A^n)$, comprised of those
elements of augmentation~$1$, can be taken to compute the
cohomology. In turn, this is useful because in positive
characteristic~$S(R[G])$ can sometimes be very simple, and indeed
isomorphic to a sum of several copies of the abelian group
underlying~$R$. In the sequel we shall prove a much more precise
statement (Theorem~\ref{thm-decomposition-exp}), but for the present
discussion let us be content with the following.

\begin{lem}
Let~$G= (\z/2)^r$, and let~$R$ be any commutative ring of
characteristic~$2$. Put 
\[ S(R[G]) = \{ x \in R[G] : \varepsilon (x) = 1 \} \, ,   \]
where~$\varepsilon $ is the augmentation. Then there is an isomorphism 
\[ S(R[G]) \cong R^{\oplus 2^r - 1} \, .   \]
\end{lem}

(In order to avoid confusion with other upperscripts, we
write~$R^{\oplus 2^r - 1}$ for a direct sum of~$2^r - 1$ copies
of~$R$.)

\begin{proof}
By induction on~$r$. For~$r=1$ it is immediate that, writing~$\z/2 =
\{ 1, \sigma  \}$, the map 
\[ a \mapsto 1 + a + a \sigma   \]
is the required isomorphism between~$R$ and~$S(R[\z/2])$. 

Now suppose~$G = H \times \z/2$ with~$H= (\z/2)^{r} $. It is clear
that 
\[ S(R[G]) \cong S(R[H]) \times K  \]
where~$K$ is the group of those~$x \in S(R[G])$ mapping to~$1$ under
the map~$S(R[G]) \to S(R[H])$, itself induced by the projection~$G \to
H$. However we can view~$R[G]$ as~$R'[\z/2]$ with~$R'= k[H]$, and
under this identification, $K$ is simply~$S(R'[\z/2])$.

By induction, $S(R[H])$ is isomorphic to~$R^{\oplus 2^{r-1} - 1}$,
while the case~$r=1$ just treated shows that~$S(R'[\z/2])$ is
isomorphic to~$R'$, itself clearly isomorphic to~$R^{\oplus 2^r}$ as
abelian group. The result follows.
\end{proof}

So at least in the case of the group~$G= (\z/2)^r$, we are led to
study a cochain complex which in degree~$n$ is made of~$2^r - 1$
copies of the additive abelian group~$k[G^{n-1}] = A^{n-1}$.

It is of course tempting to compare this cochain complex to another
one having the same underlying abelian groups: indeed
in~\S\ref{subsec-defs-twist} we proved that the ``additive''
cosimplicial group~$A^*$ had zero cohomology. In principle, the
differential on~$S(A^*)$ must be related to that on~$A^*$, but in
order to make this statement precise we need a much more explicit
description of the isomorphism in the Lemma.

Eventually we shall do just that, proving that the cohomology
of~$S(A^*)$ vanishes in degress~$\ge 3$ (though not below).

In the rest of this section, we explain in detail a two-step reduction
process for the computation of the cohomology of cosimplicial groups:
first the restriction to the kernel of the last codegeneracy, and then
in good cases a second, very similar restriction. In brief, expressive
terms, when this two-step reduction is performed for both~$A^*$
and~$A^*_\mm$, the complexes we get are almost the same.

\subsection{Basic fact}

We shall elaborate on the following trivial lemma in
homological algebra.

\begin{lem} \label{lem-iterate} Let~$(C^*, d)$ be a cochain complex of
  abelian groups. Assume that
\[ d = \sum_{i=0}^{n+1} (-1)^i d^i \, ,   \]
where~$d^i\colon C^n \to C^{n+1}$ is a homomorphism (we do not assume
the cosimplicial identities!).

Assume that there are maps~$\varepsilon_n \colon C^n\to C^{n-1}$
satisfying
\[ \varepsilon_{n+1}(d^i(x)) = d^i(\varepsilon_{n}(x))
\qquad \textnormal{for}~ 0\le i < n \, , 
\]
and 
\[ \varepsilon_{n+1}(d^n(x)) = \varepsilon_{n+1}(d^{n+1}(x)) = x \, .  \]
Put~$K^n= \ker \varepsilon_n$. Then~$d$ carries~$K^n$ into~$K^{n+1}$,
and we have 
\[ H^n(C^*, d) = H^n(K^*, d) \, .   \]

\end{lem}

\begin{proof} 
The fact that~$d$ carries~$K^n$ into~$K^{n+1}$ is trivial.

Observe the following: if~$x\in C^n$ is any element such that~$\varepsilon
(d(x)) = 0$, then~$d(\varepsilon (x)) = \pm d^n(\varepsilon (x))$, and in
particular~$d^n(\varepsilon (x))$ is a coboundary.

Applying this to an~$x\in C^n$ such that~$d(x) = 0$, we see that the
cohomology class of~$x$ is the same as that of~$x'= x -
d^n(\varepsilon (x))$. However~$\varepsilon_n (x') = \varepsilon_n
(x)- \varepsilon_n d^n \varepsilon_n(x) = \varepsilon_n(x)-
\varepsilon_n(x) = 0$, that is~$x'\in K^n$. Thus the natural map 
\[ H^n(K^*, d) \to H^n(C^*, d)  \]
is surjective.

To see that it is injective, too, pick~$x\in K^n$ such that~$x = d(y)$
for some~$y \in C^{n-1}$. Since~$\varepsilon (x)= 0 = \varepsilon
(d(y))$, the observation above applied to~$y$ shows
that~$d^{n-1}(\varepsilon_{n-1} (y))$ is a coboundary. Therefore if we
put~$y' = y - d^{n-1}(\varepsilon_{n-1} (y))$, we have~$d(y')=
d(y) = x$. However~$\varepsilon_{n-1} (y') = \varepsilon_{n-1} (y) -
\varepsilon_{n-1} d^{n-1} (\varepsilon_{n-1} (y)) = 0$, so~$y'\in
K^{n-1}$, and~$y$ is a coboundary in the complex~$K^*$.
\end{proof}

This lemma allows us to replace the cochain complex~$(C^*, d)$ by a
smaller complex, without losing the cohomological information. The
purpose of the next subsection is to show that, if~$(A^*, d)$ is the
complex associated to a cosimplicial group, then in the vein of the
above lemma we may produce a subcomplex~$(B^*, d)$ which computes the
cohomology of~$(A^*, d)$. However this time there is finer information
available on the coboundary of~$B^*$ (see
Lemma~\ref{lem-diff-explicit}). What is more, the shifted
complex~$(C^*, d) := (B^{*+1}, d)$ retains enough of the original
cosimplicial structure for us to apply Lemma~\ref{lem-iterate} in good
cases. We are thus capable of making a two-step reduction from~$A^*$
to~$C^*$ to~$K^*$.

\subsection{Reduction of cosimplicial abelian groups} \label{subsec-assumptions}

We assume that~$A^*$ is a cosimplicial abelian group, written
additively for now. We write~$\varepsilon_n \colon A^n \to A^{n-1}$
for the degeneracy map~$s^{n-1}$, in order to make a parallel with
Lemma~\ref{lem-iterate}. We recall that~$\varepsilon_{n+1}(d^i(x)) =
d^i(\varepsilon_n(x))$ for~$0 \le i < n$,
while~$\varepsilon_{n+1}(d^n(x)) = \varepsilon_{n+1}(d^{n+1}(x)) =
x$. These follow from the cosimplicial identities (which are recalled
in the Appendix). In particular, the map $d^{n+1}\colon A^n \to
A^{n+1}$ is injective.

We define~$\beta^i \colon A^{n-1} \to A^n$ by~$\beta^i= d^i$ for~$0
\le i \le n$, and~$\beta^{n+1} = \beta^n = d^n$. Note that in this way we
have created~$n+2$ maps out of~$A^{n-1}$, and for each of them there
is a commutative diagram 
\[ \begin{CD}
A^{n-1} @>{\beta^i}>> A^n \\
@V{d^n}VV                  @VV{d^{n+1}}V \\
A^n   @>{d^i}>> A^n \, . 
\end{CD}
\]
Indeed, for~$0\le i \le n$ the commutativity follows from the
cosimplicial relations, while for~$i=n+1$ it is tautological. Let us
write~$\iota_{n+1} = d^{n+1} \colon A^n \to A^{n+1}$.

We now further define a map~$\beta = \beta_n \colon A^{n-1} \to A^n$ by
\[ \beta = \sum_{i=0}^{n+1} (-1)^i \beta^i = \sum_{i=0}^{n-1} (-1)^i
\beta^i\, , \]
the equality following from~$\beta^n = \beta^{n+1}$. We have~$d\circ
\iota = \iota \circ \beta $.

Since~$\iota_n$ is injective, for all~$n$, we deduce from~$d^2 = 0$
that~$\beta^2 = 0$.

Thus we have produced a new cochain complex~$(m(A)^*, \beta )$,
where~$m(A)^n = A^{n-1}$.  Moreover~$\iota $ is a cochain map~$m(A)^*
\to A^*$. As a result there is an induced map in cohomology
\[ H^n(m(A)^*, \beta ) \to H^n(A^*, d) \, .   \]

\begin{lem}
  Under these conditions, $\varepsilon_*\colon (A^*, d) \to (m(A)^*,
  \beta )$ is a map of cochain complexes.
\end{lem}

\begin{proof}

\[ \varepsilon_{n+1}\left( \sum_{i=0}^{n+1} (-1)^i d^i(x)  \right) =
\sum_{i=0}^{n-1} (-1)^i d^i(\varepsilon_{n}(x)) = \beta (\varepsilon_n(x)) \,
.  \]
This also works for~$n=0$.
\end{proof}

\label{subsec-decomposition}

Now for each~$n \ge 1$ we have~$\varepsilon_n \circ \iota_n = id$. It
follows that~$p_n= \iota_i \circ \varepsilon_n$ is a projector that
commutes with the coboundary maps; we may write~$A^n = A^{n-1} \oplus
B^n$ with~$B^n = \ker(p_n)$, and~$d$ carries~$B^n$ into~$B^{n+1}$.

In other words there is a direct sum of cochain complexes~$A^* =
m(A)^* \oplus B^*$, and in cohomology we get 
\[ H^n(A^*, d) = H^n( m(A)^*, \beta ) \oplus H^n(B^*, d) \, .  \]

However under these assumptions we can also show:

\begin{lem}
For all~$n$ we have~$H^n(m(A)^*, \beta ) = 0$.
\end{lem}

\begin{proof}
Let~$x\in A_{n-1}$ be a cocycle in degree~$n$. The condition~$\beta
(x) = 0$ reads 
\[ \sum_{i=0}^{n-1} (-1)^i d^i(x) = 0 \, .   \]
By adding~$(-1)^nd^n(x) = (-1)^n \iota_{n}(x)$ on each side, we find
that~$d(x) = (-1)^n\iota_{n}(x)$. Now applying~$\varepsilon_n$ to this
equality yields 
\[ \varepsilon (d(x)) = \beta (\varepsilon (x)) = (-1)^n \varepsilon (\iota
(x)) = (-1)^n x \, .   \]
Hence~$x$ is the coboundary of~$(-1)^n \varepsilon (x)$.
\end{proof}

Hence:

\begin{prop} \label{prop-coho-iso} 
When~$A^*$ is a cosimplicial abelian group, then for all~$n\ge 0$
there is an isomorphism
\[ H^n(A^*, d) \cong H^n(B^*, d) \, ,   \]
where~$B^n$ is the kernel of~$\varepsilon_n$. 
\end{prop}

For computational purposes, the following expressions will help
dealing with the differential on~$B^*$. As observed, the
differential~$d$ of the complex~$A^*$ carries~$B^*$ into itself, but
the same {\em cannot} be said of the individual coface
maps~$d^i$. Instead, we have the following formulae. Let~$q=q_n= id -
p_n$ be the projector orthogonal to~$p_n$, which is a projector
onto~$B^n$, and let~${\bar d}^i = q\circ d^i \colon B^n \to B^{n+1}$;
from the relation~$q\circ d(x) = d(x)$ for~$x\in B^n$ we certainly
have
\[ d(x) = \sum_{i=0}^{n+1} (-1)^n {\bar d}^i(x) \qquad\textnormal{for}~x \in
B^n \, . 
\]
This relation will also be clear from the following more precise
equations. 

\begin{lem} \label{lem-diff-explicit}
We have 
\[ {\bar d}^i = d^i \qquad ~\textnormal{for}~ 0\le i < n \, , 
\]
while 
\[ {\bar d}^n = d^n - d^{n+1} \, ,   \]
and 
\[ {\bar d}^{n+1} = 0 \, .   \]
\end{lem}

\begin{proof}
For~$x\in B^n$ we have 
\[ {\bar d}^i(x) = d^i(x) - \iota_{n+1} \varepsilon_{n+1}(d^i(x)) \, .   \]
We have~$\varepsilon_{n}(x)=0$ by definition of~$B^n$, so from
assumption two we get the formula in the case~$0\le i < n$.

For~$i=n$, we use~$\varepsilon_{n+1}(d^n(x)) =
\varepsilon_{n+1}(d^{n+1}(x)) = x$ from assumptions two and
three. Since~$\iota_{n+1}(x) = d^{n+1}(x)$, we do have~${\bar d}^n = d^n -
d^{n+1}$.

The case~$i= n+1$ is similar.
\end{proof}

The fact that~${\bar d}^{n+1}$ is the zero map encourages us to
consider~$B^n$ as being in degree~$n-1$, that is, to consider the
complex~$(C^*, d)= (B^{*+1}, d)$. As announced, in practice we will be
able to apply Lemma~\ref{lem-iterate} to~$(C^*, d)$, though we will
not try to look for axioms on~$(A^*, d)$ for this to hold in
general. Let us give an example at once.

\subsection{First application}

Let~$G = \left( \z/2 \right)^r$, and let~$A^n = k [G^n]$ as in
\S\ref{subsec-defs-twist}; these comprise a cosimplicial abelian group
whose differential will be denoted by~$d$. Its cohomology is zero by
Lemma~\ref{lem-goodwillie}. We shall apply
Proposition~\ref{prop-coho-iso} and deduce the existence of certain
cochain complexes with zero cohomology.

Here and elsewhere, we shall use the following notation: for~$\sigma
\in G$, we write~$\sigma_n$ for the element 
\[ (1, 1, \ldots, 1, \sigma ,1, \ldots, 1) \in G^N  \]
with~$\sigma $ in the~$n$-th position, for some~$N \ge n$ which is
always clear from the context. (Usually~$N=n$.)

\begin{prop} \label{prop-application-homology}
Let~$\sigma \in G$. Define a cochain complex~$(A^*, \delta_\sigma ) $
with~$A^*$ as above and~$\delta \colon A^{n-1} \to A^n$ given by 
\[ \delta_\sigma (a) = d(a) + a(1+\sigma_n) \, .   \]
Then~$\delta_\sigma \circ \delta_\sigma = 0$ and~$H^n(A^*,
\delta_\sigma ) = 0$ for~$n \ge 0$. Moreover, the subcomplex~$\bar
A^*$ of elements of having zero augmentation is preserved
by~$\delta_\sigma $ and we also have~$H^n(\bar A^*, \delta_\sigma ) =
0$ for~$n \ge 0$.
\end{prop}   

\begin{proof}
We apply Proposition~\ref{prop-coho-iso}. We have a
decomposition~$A^{n} = A^{n-1} \oplus B^n$ where~$B^n$ is the kernel
of the augmentation~$A^n \to A^{n-1}$, and the cohomology of~$B^n$
vanishes. 

For~$\sigma \in G$ such that~$\sigma \ne 1$, let~$A^n_\sigma $ be a
copy of~$A^n$. For~$n \ge 1$ we use the identification 
\[ \begin{array}{rcl}
\psi \colon {\displaystyle \bigoplus_\sigma A^{n-1}_\sigma} & \stackrel{\cong}{\longrightarrow} & B^n \\
                (a_\sigma )_\sigma & \mapsto & {\displaystyle  \left( \sum_\sigma
a_\sigma    \right) + \sum_\sigma  a_\sigma \, \sigma_n} \, .  
\end{array}   \]
This isomorphism defines a differential~$\delta = \psi^{-1} \circ
d \circ \psi$ on~$C^* = \bigoplus_\sigma A^*_\sigma $ which must then
satisfy~$H^n(C^*, \delta ) = H^n(B^{*+1}, d) = 0$ at least for~$n \ge
1$. 

Now checking the definitions, we see that~$\delta $ splits as the
direct sum of the differentials~$\delta_\sigma$ given in the statement of
the Proposition. It is immediate that the cohomology of~$\delta_\sigma
$ is also zero in degree~$0$. We have proved the first statement.

Finally, since~$\varepsilon_n (1 + \sigma_n) = 0$, we can apply
Lemma~\ref{lem-iterate} to~$(A^*, \delta_\sigma )$. The second
statement follows.
\end{proof}



\section{The case of~$\z/2$} \label{sec-zedmodtwo}

In this section we compute completely the Sweedler cohomology
of~$\O_k(\z/2)$, or equivalently the twist cohomology
of~$k[\z/2]$. The group with two elements is so simple that other
approaches than the one below are possible, which could be easier
(looking at the normalized cocycles is a good idea). However, we
choose to apply the strategy described in~\S\ref{subsec-strategy} as
an illustration which is much less technical than the general case.

\subsection{The unit sphere} \label{subsec-log-exp}

Let~$R$ be a commutative ring of characteristic~$2$. The elements of
the group~$\z/2$ will be written~$1$ and~$\sigma $. The group
algebra~$A= R[\z/2]$ consists, of course, of the elements~$z = x + y
\sigma $ with~$x, y\in R$.

We define the {\em modulus} of~$z$ to be 
\[ |z| = x + y \in R\qquad (= \sqrt{(x+y)^2} = \sqrt{x^2 + y^2}) \, ,  \]
or in other words we shall write~$|z|$ for the augmentation of~$z$.
We note that~$z\mapsto |z|$ is a map of algebras~$A\to R$.

We have the relation
\[  z^2 = |z|^2 \, ,   \]
from which it follows that~$z$ is invertible in~$A$ if and only
if~$|z|$ is invertible in~$R$ (and then~$z^{-1} = |z|^{-2} z$). As a
result the elements in the {\em unit sphere} 
\[ S(A) = \{ z\in A ~:~ |z| = 1 \}
\]
are all invertible in~$A$.

 There is an isomorphism
\[ A_\mm \stackrel{\simeq}{\longrightarrow} R_\mm \times S(A) \,
, 
\]
given by~$z \mapsto (|z|,  \frac{z} {|z|})$.

An element in~$S(A)$ is of the form~$(1+x) + x \sigma $. For any~$x\in
R$ we define its {\em exponential} to be precisely 
\[ e^x  = (1+x) + x \sigma \in A_\mm \, .   \]
There is the usual formula 
\[ e^{a+b} = e^a e^b \, .   \]
The exponential gives an isomorphism~$R \to S(A)$, whose inverse we
call the logarithm and write~$\log\colon S(A)\to R$. We end up with an
isomorphism 
\[ A_\mm \stackrel{\simeq}{\longrightarrow} R_\mm \times R \, , 
\]
given by~$z\mapsto (|z|, \log( \frac{z} {|z|} ))$.

\subsection{Higher group algebras}

Let~$k$ be a ring of characteristic~$2$, let~$A^0 = k$, and for~$n\ge
1$ let~$A^n = k[(\z/2)^n]$. We always see~$A^n$ as a subring
of~$A^{n+1}$. The evident generators for~$(\z/2)^n$ will be
written~$\sigma_1, \ldots, \sigma_n$, so that for example~$\sigma_1
\sigma_2$ is an element of~$A^2$. Of course it may also be considered
as an element of~$A^3$, but in practice the ambiguities created are of
no consequence.

It is fundamental that~$A^{n+1} = A^n[\z/2]$.

\begin{prop} \label{prop-multiplicative}
 Let~$A^*$ and~$d$ be as above. Define a cochain complex~$(A^{*},
  \partial )$ with~$\partial \colon A^{n-1} \to A^n$ given by
\begin{align*}
\partial (a) & = d(a) + d^{n}(a)(1 + e^{1+a}) \\
          & = d(a) + (a+a^2) (1+ \sigma_n) \, . 
\end{align*}
Then~$\partial \circ \partial = 0$ and~$H^n( A^*_\mm, d) = H^{n-1}(
A^{*}, \partial )$ for~$n \ge 2$. Moreover, the cochain
complex~$(A^*, \partial )$ satisfies the hypotheses of
Lemma~\ref{lem-iterate}. 
\end{prop}

\begin{proof}

We apply Proposition~\ref{prop-coho-iso}. The decomposition~$A^n_\mm =
A^{n-1}_\mm \times B^n$ as in \S\ref{subsec-decomposition} can be
identified for~$n\ge 1$ with the decomposition~$A^{n}_\mm =
A^{n-1}_\mm \times S(A^n)$. Using the logarithm, we deduce the
decomposition $A^{n}_\mm = A^{n-1}_\mm \times A^{n-1}$, and
Proposition~\ref{prop-coho-iso} now states that there is an
isomorphism for~$n\ge 2$
\[ H^n( A^*_\mm, d) = H^n( A^{*-1}, \partial ) \, ,   \]
where~$\partial $ needs to be explicitly described.

This is based on Lemma~\ref{lem-diff-explicit}, and we write~${\bar
  d}^i$ for the maps described there -- keeping in mind that we need
to use multiplicative notation now. We point out that the elements
of~$B^n = S(A^n) \cong A^{n-1}$ are of order~$2$, so we may ignore the
inverses, just like we can ignore the signs in additive notation.

We write~$\bar \partial ^i(a) = \log( {\bar d}^i(e^a)) $, so that 
\[ \partial (a) = \sum_{i=0}^{n+1} (-1)^i \bar \partial ^i (a) \, .    \]
(Again the signs are here for decoration.)

For~$0\le i < n$, and~$a\in A^{n-1}$, we check readily
that~$d^i(1+a + a \sigma_n) = 1 + d^i(a) + d^i(a)
\sigma_n$, which reads~$d^i(e^a) = e^{d^i(a)}$. It follows that~$\bar
\partial ^i(a) = d^i(a)$ in these cases.

For~$i=n$, we first compute 
\[ d^n(e^a) = 1 + a + a  \sigma_n \sigma_{n+1} \, ,   \]
and
\[ d^{n+1}(e^a) = 1 + a  + a \sigma_n  \, .   \]
The product (=quotient) of these is 
\begin{multline*}
1 + a^2  + (a + a^2) \sigma_n  + a^2 \sigma_{n+1} + (a +
a^2)  \sigma_n  \sigma_{n+1} \\
= e^{a^2 + (a + a^2) \sigma_n} = e^{a(a + (1 + a) \sigma_n)} =   e^{d^n(a)e^{1+a}} \,
. 
\end{multline*}
Thus~$\bar \partial^{n}(a) = d^n(a)e^{1+a} = d^n(a) + d^n(a)(1 +
e^{1+a})$. And~${\bar d}^{n+1}=0$ implies~$\bar \partial^{n+1} = 0$, of
course.

This gives the expression for~$\partial(a)$. To show that the
hypotheses of lemma~\ref{lem-iterate} are satisfied, we note that,
this time, $(A^*, \partial)$ is obtained from~$(A^*, d)$ by replacing
the last coface~$d^{n+1}(a) = a\otimes 1$ by~$d^{n+1}(a)(e^{1+ a})$;
however $\varepsilon (e^{1+a}) = 1$ so~$\varepsilon ( d^{n+1}(a)(e^{1+
  a})  ) = \varepsilon (d^{n+1}(a)) = a$, as we wanted.
\end{proof}

Comparing the Propositions~\ref{prop-application-homology}
and~\ref{prop-multiplicative} shows how close the
differentials~$\delta_\sigma  $ and~$\partial$ are: the expression~$a + a^2$
simply replaces~$a$, so that they are the same ``at first order''.

\subsection{Sweedler cohomology of~$\O_k(\z/2)$}

\begin{thm} \label{thm-main-zedmodtwo}
Let~$k$ be a ring of characteristic~$2$. The twist cohomology
of~$k[\z/2]$, or the Sweedler cohomology of~$\O_k(\z/2)$, is given by
\[ H^n_{sw}(\O(\z/2)) = \left\{ \begin{array}{l}
0 ~\textnormal{for}~ n\ge 3 ~\textnormal{or}~ n=0 \, , \\
~\\
k / \{ x + x^2 ~|~ x \in k \} ~\textnormal{for}~ n=2 \, , \\
~\\
\z/2 ~\textnormal{for}~ n=1 \, . 
\end{array}\right.  \]
In particular when~$k$ is algebraically closed then these groups
vanish in degrees~$\ge 2$.
\end{thm}

\begin{proof}
The statements for~$n=0$ or~$n=1$ are (easy) general facts. We first
prove that~$H^n(A^*_\mm, d) = 0$ for~$n \ge 3$, which is the first case
above. 

We have seen (Proposition~\ref{prop-application-homology}) that for~$n
\ge 2$
\[ 0 = H^n(A^*, d) = H^{n-1}(A^*, \delta_\sigma  ) \, .  \]
Moreover, by applying Lemma~\ref{lem-iterate} we have 
\[ H^{n-1}(A^*, \delta ) = H^{n-1}(K^* , \delta ) \, ,   \]
where~$K^n$ is the subgroup of elements~$a\in A^n$ such
that~$\varepsilon_n(a) = 0$. (This is also part of the conclusion of
Proposition~\ref{prop-application-homology} where~$K^*$ is
denoted~$\bar A^*$.)

On the other hand we have also (Proposition~\ref{prop-multiplicative})
for~$n \ge 2$
\[ H^n(A^*_\mm, d) = H^{n-1}(A^*, \partial ) \, .  \]
By applying Lemma~\ref{lem-iterate} we have 
\[ H^{n-1}(A^*, \partial ) = H^{n-1}(K^* , \partial ) \, ,   \]
where~$K^n$ is precisely the same as above.

Now the fundamental observation is that~$a^2 = \varepsilon_n(a)^2$
for~$n\ge 1$. So for an element~$a\in K^n$, we have~$a^2 = 0$. As a
result, {\em the differentials~$\delta $ and~$\partial$ agree
  on~$K^*$}, from~$K^1$ and above. Thus~$H^r(K^*, \delta ) =
H^r(K^*, \partial)$ for~$r\ge 2$.

It follows that
\[ H^{n-1}(A^*, \delta ) \cong H^{n-1}(A^*, \partial) \, ,   \]
for~$n\ge 3$, whence the result.

Now we turn to the computation of~$H^2(A^*_\mm, d) = H^1(A^*,
\partial)$. For any element~$ a + b \sigma_1 \in A^1$, we compute that 
\[ \partial(a + b \sigma_1) = (a^2 + b^2) + (a+b + a^2 + b^2) \sigma_2 \in A^2 \, ,   \]
so the kernel of~$\partial$ in degree~$1$ is isomorphic to~$k$, and is
comprised of those elements of the form~$ a + a \sigma_1$. On the
other hand for~$x\in k = A^0$, we have
\[ \partial(x) = (x+x^2) + (x+x^2) \sigma_1 \, .   \]
This shows the announced result for~$n = 2$. When~$k$ is algebraically
closed, note that the equation~$x^2 + x = a$ always has a solution
regardless of the parameter~$a\in k$, so~$H^2$ vanishes, too.
\end{proof}

\begin{ex} \label{ex-zedmodtwo}
Let us illustrate the theorem with a simple example. Let~$q = 2^r$,
and take~$k= \f_{q}$, the field with~$q$ elements. The map~$\f_q \to
\f_q$ sending~$x$ to~$x + x^2$ has kernel~$\f_2$, so its cokernel has
dimension~$1$ over~$\f_2$. Thus~$H^2_{tw}(\f_q[\z/2]) = \f_2$.

The non-trivial element is described as follows. There is a non-zero
element~$a \in \f_q$ which is not of the form~$x + x^2$, and~$a +
a \sigma_1 $ is a representative of the non-zero class in~$H^1(A^*,
\partial)$. Via the isomorphism with~$H^2_{tw}( \f_q[\z/2])$, we obtain the
twist 
\[ F = e^{a + a \sigma_1 } = (1 + a) 1\otimes 1 + a \sigma_1\otimes 1
+ a 1 \otimes \sigma_2 + a \sigma_1 \otimes \sigma_2 \in \f_q[\z/2]^{\otimes 2}\, .  \tag{*}  \]
It is symmetric, that is~$F = F_{21}$ in common Hopf-algebraic
notation, so that the element~$R_F = F_{21} F^{-1} = 1 \otimes
1$. (Whenever~$F$ is a twist, the element~$R_F$ is always an
``$R$-matrix'', of which much more in the rest of this paper, and it
normally holds important information about~$F$.)

There is a simple way to see~$F$ in action. Whenever~$\A$ is
an~$\f_q$-algebra endowed with a~$\z/2$-action, we can twist it
using~$F$ into a new algebra~$\A_F$. A lot of information about this
is presented in~\cite{akira}, but we will keep things elementary and
only state that if~$\mu \colon \A \otimes \A \to \A$ is the original
multiplication, then it is twisted to 
\[ x * y = \mu( x \otimes y \, F) \, .   \]
That this new multiplication is associative is equivalent to~$F$ being
a twist. It also follows from {\em loc.\ cit}.\  that a fundamental
example is~$\A = \O(\z/2)$, the algebra of functions on~$\z/2$, so let
us only look at this case.

This algebra is~$2$-dimensional over~$\f_q$, with a basis given by the
constant function~$1$ and the Dirac function~$\delta $ at the neutral element
of~$\z/2$. In~$\A_F$ the unit will be unchanged (easy check), and
there remains to compute 
\[ \delta * \delta = \mu(\delta \otimes \delta F)=  a + \delta \, .   \]
So in~$\A_F$, we have a solution of~$x^2 + x = a$, namely~$x= \delta
$. It follows that~$\A_F$ is simply~$\f_{q^2}$. 

In fact, if~$F$
corresponds to {\em any}~$a \in k$ by formula (*), we will
have~$\O(\z/2)_F = \f_q[x]/(x^2 + x + a)$ (which is two copies
of~$\f_q$ when~$a$ is already of the form~$x^2 + x$ in~$\f_q$).

We recall that in characteristic~$0$, we have~$H^2_{tw}(k[\z/2]) = H^1(k,
\z/2) = k^\times / (k^\times)^2$. If~$F$ is the twist corresponding to
the class of~$a$ modulo squares, then~$\O(\z/2)_F$ is isomorphic
to~$k[\sqrt{a}] = k[x]/(x^2 - a)$.

\end{ex}



\section{The general case} \label{sec-generic-case}

We now let~$G = \left( \z/2 \right)^r$ be any elementary
abelian~$2$-group, and we let~$k$ be any ring of
characteristic~$2$. In this section we prove that the twist cohomology
of~$k[G]$ vanishes in degrees~$\ge 3$. We also give information about
the low-dimensional cohomology groups.

\subsection{The unit sphere} \label{subsec-exponential}

Let~$R$ be a ring of characteristic~$2$. In this paragraph we seek a
describtion of 
\[ S(R[G]) = \{ x \in R[G] : \varepsilon (x) = 1 \} \, .   \]
Here~$ \varepsilon \colon R[G] \to R$ is the augmentation. It will
sometimes be convenient to write~$S_R(R[G])$ for emphasis; consider
for example the group algebra~$R[G\times G] = R'[G]$ for~$R'= R[G]$,
for which the notation~$S_{R[G]}(R[G\times G])$ refers to the
augmentation $ \varepsilon \colon R[G \times G] \to R[G]$. In fact,
the following decomposition will be useful in the sequel: 
\[ S_R(R[G \times G]) = S_R(R[G]) \times S_{R[G]}(R[G\times G]) \, .  \tag{$\dagger$} \]
The proof is immediate.

A useful device for the study of~$S(R[G])$ is the {\em
  exponential}. Namely, whenever~$u \in R[G]$ satisfies~$\varepsilon
(u) = 0$, and thus~$u^2 = \varepsilon (u)^2 = 0$, we put
\[ \exp_u(a) = 1 + au \qquad (a \in R) \, .  \]
Of course~$\exp_u(a)$ is an element of~$S(R[G])$.

\begin{lem} \label{lem-prop-exp}
The exponential enjoys the following properties:
\begin{enumerate}
\item $\exp_u(a+b) = \exp_u(a) \, \exp_u(b)$.
\item $\exp_u(a) \, \exp_v(a) = \exp_{u+v} (a) \, \exp_{uv} (a^2)$. 
\end{enumerate} 
\end{lem}

Fix once and for all a basis~$\Sigma \subset G $ for~$G$ as
an~$\f_2$-vector space. For a subset~$X \subset \Sigma $, we put 
\[ u_X = \prod_{\sigma \in X} (1 + \sigma ) \, .   \]

\begin{thm} \label{thm-decomposition-exp}
For each non-empty~$X \subset \Sigma $, let~$R_X$ be a copy of the
abelian group underlying~$R$. Then the map 
\[ \begin{array}{rcl}
{\displaystyle \myexp \colon \bigoplus_{\emptyset \ne X \subset \Sigma } R_X} & \longrightarrow &
S(R[G]) \\
 (a_X)_X             & \mapsto & {\displaystyle \prod_X \exp_{u_X}( a_X) }
\end{array}\]
is an isomorphism.
\end{thm}

\begin{ex} 
For~$G= \z/2 \times \z/2= \langle \sigma , \tau \rangle$, the Theorem
asserts that there is an isomorphism of abelian groups
\[ R^{\oplus 3} = R_{\{ \sigma, \tau  \}} \oplus R_{\{ \sigma  \}}
\oplus R_{\{ \tau  \}} \longrightarrow S(R[G])  \]
\[ (\lambda, \mu , \nu ) \mapsto \exp_{1+\sigma + \tau + \sigma \tau
}(\lambda ) \exp_{1 + \sigma } (\mu ) \exp_{1+\tau } (\nu ) \, .  \]
In fact the generic element 
\[ 1 + (a+b+c) + a \sigma + b \tau  + c \sigma \tau \in S(R[G])  \]
is of the form above with~$\lambda = 1 + c + (a+c)(b+c)$, $\mu = a+c$
and $\nu = b+c$. 
\end{ex}

Before we turn to the proof, we need some notation and a Lemma. For~$X
\subset \Sigma $, we let~$G_X$ denote the subgroup of~$G$ generated
by~$X$. We can see~$G_X$ as a quotient of~$G$ as well, the evident
map~$G \to G_X$ having kernel~$G_{\Sigma \smallsetminus X}$. Thus we
can speak of the image of an element~$x \in R[G]$ in~$R[G_X]$.

\begin{lem}
Let~$x \in S(R[G])$. Then~$x$ is of the form~$\exp_{u_\Sigma }(a)$
for~$a\in R$ if and only if its image in all the rings~$R[G_X]$
for~$\emptyset \ne X \varsubsetneq \Sigma $ is~$1$.
\end{lem}

\begin{proof}
The condition is clearly necessary, as~$u_\Sigma $ maps to~$0$ in all
the rings~$R[G_K]$. To prove that it is sufficient, we proceed by
induction on the rank of~$G$. For~$G = \z/2 = \{ 1, \sigma \}$, it is
certainly true that any~$x \in S(R[\z/2])$ must be of the form~$x= 1 + a + a
\sigma = \exp_{1+ \sigma }(a)$.

Now for the general case, write~$\Sigma = \{ \sigma  \} \cup \Sigma_0$
and~$H = G_{\Sigma_0}$, so that~$G = \z/2 \times H$. We have~$R[G]=
R'[H]$ for~$R'= R[\z/2]$, and we know that the result of the Lemma
holds for~$R'[H]$.

Let~$x$ be as in the Lemma. As an element of~$R'[H]$, the augmentation
of~$x$ (which is also its image under the map~$G \to G_{\{ \sigma
  \}}$) must be~$1$ by hypothesis; further, when viewed in~$S(R'[H])$
the element~$x$ still satisfies the hypotheses of the Lemma, so~$x =
\exp_{u_{\Sigma _0}}(a')= 1 + a' u_{\Sigma_0}$ for some~$a' \in R' =
R[\z/2]$. By looking at the image of~$x$ under~$G \to H$, which must
be~$1$, we see that~$a' = a + a \sigma $ for some~$a \in R$. We
observe that
\[ (1 + \sigma ) \, u_{\Sigma_0} = u_{\Sigma }\, ,   \]
so in the end~$x = \exp_{u_\Sigma }(a)$, as we wished to prove.
\end{proof}

\begin{proof}[Proof of Theorem~\ref{thm-decomposition-exp}]
We prove that~$\myexp$ is injective first. Let~$(a_X)_X$ be such
that~$\myexp (a_X)_X = 1$. Let~$Z$ be such that~$a_Z \ne 0$, if there
is such a~$Z$, and assume that~$Z$ has minimal cardinality with
respect to this property.

We consider the images in~$R[G_Z]$ of various elements. Of course the
image of~$\myexp (a_X)_X$ is~$1$. Let~$Y \subset \Sigma $. If~$Y$ is
not a subset of~$Z$, then~$u_Y$ maps to~$0$ in~$R[G_Z]$, so~$\exp_{u_Y} a_Y$
maps to~$1$. If~$Y \varsubsetneq Z$, then~$a_Y = 0$ by hypothesis, so
again~$\exp_{u_Y} a_Y$ maps to~$1$. Finally, for~$Y = Z$ the image
of~$\exp_{u_Z}(a_Z)$ is itself, so in the end~$\myexp (a_X)_X$
restrict to~$\exp_{u_Z}(a_Z) = 1$. This implies (easily) that~$a_Z =
0$, a contradiction showing that~$\myexp$ is injective.

We turn to the surjectivity. Let~$x \in S(R[G])$. If the image of~$x$
in~$R[G_X]$ is~$1$ for all the proper subsets~$X$ of~$\Sigma $,
then~$x$ is in the image of~$\myexp$ by the Lemma. 

Now let~$Z$ be a proper subset of~$\Sigma $ such that the image of~$x$
in~$R[G_Z]$ is not equal to~$1$, and assume that~$Z$ has minimal
cardinality with respect to this property. By the Lemma again, the
image of~$x$ in~$R[G_Z]$ is of the form~$\exp_{u_Z}(a)$ for some~$a
\in R$. Let~$x' = \exp_{u_Z}(a)$, viewed as an element of~$S(R[G])$,
and consider~$x_1 = xx'$. Its image in~$R[G_Z]$ is~$1$ by
construction. What is more, if the image of~$x$ is~$1$ in~$R[G_X]$ for
some~$X$, then~$Z$ is certainly not a subset of~$X$; so~$u_Z$ maps
to~$0$ in~$R[G_X]$ and~$x'$ maps to~$1$ there, as does~$x_1$.

Continuing, we form~$x_2$, $x_3$, etc, such that~$x_{i+1} = x_i x
^{(i+1)}$ with~$x ^{(i+1)}$ belonging to the image of~$\myexp$, and
such that~$x_{i+1}$ maps to~$1$ in~$R[G_X]$ whenever~$x_i$ does {\em
  and} for one extra subset. This process stops when some~$x_i$ maps
to~$1$ in all the rings corresponding to all the proper subsets
of~$\Sigma $, in which case~$x_i$ is in the image of~$\myexp$ as
already observed. We conclude that~$x$ is in the image of~$\myexp$,
and this map is surjective.
\end{proof}

\subsection{The map~$\phi$} \label{subsec-map-phi}

We shall now study a certain map~$\phi$ from~$R[G]$ to~$R[G\times
  G]$. It is defined as the product of the inclusion and the diagonal,
that is
\[ \phi(x) = x \, \Delta (x) \in R[G \times G] \, .   \]
Whenever~$x \in S(R[G])$, it is clear that~$\phi(x)$ lies in the
subgroup~$S_{R'}(R'[G])$ where~$R' = R[G]$, as in ($\dagger$). Thus we
shall consider~$\phi$ as a map 
\[ \phi \colon S(R[G]) \longrightarrow S(R'[G]) \, .   \]
We can apply Theorem~\ref{thm-decomposition-exp} to both~$S(R[G])$
and~$S(R'[G])$, of course. The map~$\phi$ induces, via the~$\myexp$
isomorphisms, a map~$\tilde \phi$. The latter does not preserve the
direct sum decompositions, but it is compatible with certain
filtrations. 

For each~$\ell \ge 1$, let~$S(R[G])_\ell$ denote the image of the sum of
all the~$R_X$ with~$X$ of cardinality~$\le \ell$, under the map of the
Theorem. Also, let~$S(R[G])_X$ denote the image of the sum of all
the~$R_Y$ with~$Y \subset X$.

\begin{lem} \label{lem-phi-preserves-filtrations}
We have 
\[ \phi(S(R[G])_\ell) \subset  S(R'[G])_\ell   \]
and 
\[ \phi( S(R[G])_X ) \subset   S(R'[G])_X   \, .   \]
\end{lem}

\begin{proof}
The first statement follows from the second. There is nothing to prove
if~$X = \Sigma $. If not, consider the subgroup~$G_X \subset G$
spanned by~$X$, and appeal to the naturality of all the maps in sight
with respect to the inclusion~$G_X \to G$.
\end{proof}

There is a canonical isomorphism, for~$X$ of cardinality~$\ell$,
\[ R_X \cong \frac{ {}S(R[G])_X} { S(R[G])_X \cap S(R[G])_{\ell-1} } \,
,  \]
induced by~$\exp_{u_X}$. Thus~$\phi$ induces a map 
\[ \phi_X \colon R_X \to R'_X  \]
which we wish to describe explicitly. Let~$m_X$ be the product of the
elements of~$X$.

\begin{prop} \label{prop-phiX-explicit}
Let~$a \in R_X$ be such that~$a^2 = 0$. Then 
\[ \phi_X(a) = a \, m_X \, .   \]
\end{prop}

Note that~$m_X \in G$ and~$a m_X$ is indeed an element of~$R'=
R[G]$. 

\begin{rmk} \label{rmk-computer}
It is very likely that~$\phi_X(a) = a \, m_X$ for all~$a \in R_X$
without restriction, as soon as the cardinality of~$X$ is at least~$2$
(though definitely not when~$X$ is reduced to one element). Computer
calculations have confirmed this when~~$2 \le |X| \le 7$.
\end{rmk}

\begin{proof}
The case when~$X= \{ \sigma  \}$ is both simple and important for the
general case. In this situation we have~$u = u_X = 1 + \sigma $, and
we need to consider~$\phi(x)$ for 
\[ x= \exp_u(a) = 1 + a + a \sigma \, .   \]
Direct calculation yields then 
\[ \phi(x) = 1 + (a \sigma_1 + a^2 \sigma_1 + a^2)(1 + \sigma_2) =
\exp_{1 + \sigma_2}(a \sigma_1 + a^2 \sigma_1 + a^2) \, ,   \]
were the elements of~$G \times G$ are decorated with indices. So if we
assume that~$a^2 = 0$ we have indeed
\[ \phi(\exp_{u_X}(a)) =  \exp_{u_X}(a \sigma ) \tag{*} \]
where~$u_X$ is interpreted (slightly) differently on either side of
this equation.

We turn to the general case. Let~$X = \{ x_1, \ldots, x_\ell \}$, and
let~$s_i$ be the~$i$-th symmetric function in the~$x_j$'s (so
that~$s_\ell = m_X$). We have~$u = u_X = 1 + s_1 + \cdots + s_\ell$. The idea is
to replace~$u$ by~$1 + s_\ell$ and reduce to the case~$\ell=1$. 

To see this, start by observing that 
\[  \exp_u (a) \, \exp_{1+s_\ell} (a) = \exp_{s_1 + \cdots + s_{\ell-1}} (a)
\, , \]
from Lemma~\ref{lem-prop-exp} (2) (either since~$a^2=0$ or since~$u(1
+ s_\ell) = 0$). Call~$y$ the right hand side. By a repeated use of
Lemma~\ref{lem-prop-exp} (2), we see that we may write~$y$ as a
product of terms of the form~$\exp_v (a)$ with~$v$ in
the~$\f_2$-subalgebra of~$R[G]$ generated by some of the~$x_i$'s, but
always less than~$\ell$ of them, so~$y$ lies in~$S(R[G])_{\ell-1}$. Finally
\[ \exp_u(a) = \exp_{1 + s_\ell}(a) ~\textnormal{mod}~ S(R[G])_{\ell-1} \, .  \tag{**} \]

On the other hand we can compute the value of~$\phi( \exp_{1 + s_\ell}
(a))$ by (*) (applied to the case~$X = \{ s_\ell \}$):
\[ \phi( \exp_{1+s_\ell}(a)) = \exp_{1+s_\ell} (a s_\ell)  \, .   \]
(Here the cautious reader can rewrite this with indices~$s_{\ell, 1}$
and~$s_{\ell,2}$ if she wishes to distinguish between the two.) Working
with (**) backwards yields the result.
\end{proof}

We conclude with some remarks about the compatibility of~$\phi$ with
augmentation maps. We write~$\varepsilon \colon R[G] \to R $ for the
usual augmentation. Keeping the notation~$R' = R[G]$, we point out
that the construction of~$R[G]$ is natural in~$R$, so
that~$\varepsilon \colon R' \to R$ induces a map~$e \colon R'[G] \to
R[G]$. (If we think of~$R'[G]$ as~$R[G]\otimes_R R[G]$, then we
have~$e(\sigma \otimes \tau ) = \tau$.) It is immediate that

\begin{lem} \label{lem-phi-augmentation}
For all~$x \in R[G]$, one has $e\circ \phi (x) = x$.
\end{lem}

We summarize the notation of this section in a commutative diagram. 
\[ \begin{CD}
{\displaystyle \bigoplus_X R_X} @>{\myexp}>> S(R[G]) \\
@V{\tilde \phi}VV           @VV{\phi}V \\
{\displaystyle \bigoplus_X R'_X} @>{\myexp}>> S(R'[G]) \\
@V{\oplus \varepsilon }VV       @VV{e}V\\
{\displaystyle \bigoplus_X R_X} @>{\myexp}>> S(R[G]) 
\end{CD}
  \]
The horizontal maps are isomorphisms, and the compositions of the
vertical maps are identities.

\subsection{Vanishing of the twist cohomology}

\begin{thm} \label{thm-big-mama-vanishing}
For~$n \ge 3$, we have~$H^n(A^*_\mm, d) = 0$.
\end{thm}

\begin{proof}
We apply Proposition~\ref{prop-coho-iso}. It follows that $H^n(A^*_\mm, d)
= H^n(B^*, d)$ where~$B^* = S_{A^{*-1}}( A^* )$ in the notation
above. Put
\[ C^* =  \bigoplus_X A_X^* \, ,   \]
where~$A^*_X$ is a copy of~$A^*$, so that the~$\myexp$ isomorphism of
Theorem~\ref{thm-decomposition-exp} provides us with an
isomorphism~$C^* \cong B^{*+1}$. We set~$\partial = \myexp^{-1} \circ
d \circ \exp$, so that our goal is to prove that~$H^n(C^*, \partial) =
0$ for~$n \ge 2$.

The differential~$\partial$ on~$C^{n-1} \cong B^n$ is the sum of
maps~$\partial^i = \myexp^{-1} \circ d^i \circ \myexp$ for~$0 \le i
\le n$ which are given by Lemma~\ref{lem-diff-explicit}. For~$0 \le i
< n$, it turns out that~$\partial^i$ coincides with~$d^i$, or rather a
direct sum of copies of~$d^i$ indexed by the subsets~$X$. For~$i=n$,
we have~$\partial^n= \myexp^{-1} \circ \phi \circ \myexp = \tilde
\phi$ in the notation of \S\ref{subsec-map-phi}, where~$R= A^{n-1}$. 

From Lemma~\ref{lem-phi-preserves-filtrations}, it follows that~$C^*$
has a filtration by subcomplexes (preserved by~$\partial$), and the
subquotients are cochain complexes of the form~$(A^*_X,
\partial_X)$. In degree~$n$ the underlying abelian group is~$A^n_X$, a
copy of~$A^n$, and the differential~$\partial_X: A^{n-1}_X \to A^n_X$
is given by
\[ \partial_X = \sum_{i=0}^{n-1} d^i + \phi_X \, , \tag{*} \]
where the notation~$\phi_X$ is as in
Proposition~\ref{prop-phiX-explicit} (again for~$R= A^{*-1}$). We will
prove that~$H^n(A^*_X, \partial_X) = 0$ for~$n \ge 2$ (and for
each~$X$), which implies the Theorem from the long exact sequences in
cohomology.

From Lemma~\ref{lem-phi-augmentation}, we see that the
complex~$(A^*_X, \partial_X)$ satisfies the hypotheses of
Lemma~\ref{lem-iterate}. As a result, in order to compute its
cohomology, we may restrict to the subgroup~$\bar A^*_X$ of those
elements~$a$ with~$\varepsilon (a) = 0$. In degree~$n \ge 1$, we
have~$a^2 = \varepsilon (a)^2$, so these elements satisfy~$a^2 = 0$.
(Recall that in degree~$0$ we have~$A^0_X = k$ and the
``augmentation'' is the zero map, so we can draw no such conclusion).

We can thus use Proposition~\ref{prop-phiX-explicit}. Together with
(*), it implies for~$a \in \bar A^n_X$ and~$n \ge 1$ that 
\[ \partial_X(a) = d(a) + a(1 + m_X) \, .   \]
In other words, in degrees~$n \ge 1$, the differential~$\partial_X$
on~$\bar A^*_X$ coincides with~$\delta_{m_X}$ considered in
Proposition~\ref{prop-application-homology}. By that Proposition, the
cohomology does vanish in degrees~$\ge 2$.  
\end{proof}

\subsection{The Sweedler cohomology groups}

In this section we prove the following result.

\begin{thm} \label{thm-main-elemab}
Let~$k$ be a field of characteristic~$2$. The twist cohomology
of~$k[\left(\z/2 \right)^r]$, or the Sweedler cohomology
of~$\O_k(\left(\z/2 \right)^r)$, is given by
\[ H^n_{sw} (\O_k(\left(\z/2 \right)^r)) = \left\{ \begin{array}{l}
0 ~\textnormal{for}~ n\ge 3 ~\textnormal{or}~ n=0 \, , \\
~\\
\left(\z/2 \right)^r ~\textnormal{for}~ n=1 \, , \\
~\\
\left( k / \{ x + x^2 : x \in k \} \right)^{\oplus r}
~\textnormal{for}~ n=2 \, . 
\end{array}\right.  \]
When~$k$ is an algebraically closed field, we have in particular $H^2_{sw}(\O_k(\left(\z/2 \right)^r)) = 0$.

\end{thm}

The statement for~$n \ge 3$ is Theorem~\ref{thm-big-mama-vanishing},
while the statements for~$n= 0$ or~$1$ are classical and easy (they
hold for any finite group~$G$). What remains is the result for~$n=2$,
which we have established in the case~$r=1$ with
Theorem~\ref{thm-main-zedmodtwo}.

Fortunately there is a Künneth-type theorem for~$H^2_{sw}$,
established in the context of ``lazy cohomology'' by Bichon and
Carnovale: see Theorem 4.8 in~\cite{bichon} which stipulates that 
\[ H^2_\ell (A \otimes B) \cong H^2_\ell(A) \times H^2_\ell(B) \times
\mathcal{ZP} (A \otimes B) \, ,   \]
for any two Hopf algebras~$A$ and~$B$. Recall that~$H^2_\ell(A) =
H^2_{sw}(A)$ when~$A$ is cocommutative. Moreover, Lemma 4.9 in {\em
  loc.\ cit}.\ describes~$\mathcal{ZP}(A \otimes B)$ as a group of Hopf
algebra homomorphisms~$A \to B^*$ (the dual of~$B$), satisfying
certain conditions. However, in our situation the following must be
noticed. 

\begin{lem} \label{lem-hopf-hom-O-kP}
Let~$G$ be any finite group, let~$P$ be a finite~$p$-group, and
let~$k$ be a field of characteristic~$p$. Then there is only one
homomorphism of Hopf algebras 
\[ \O_k(G) \longrightarrow k[P] \, ,   \]
namely the ``augmentation''~$f \mapsto f(1)1$.
\end{lem}

\begin{proof}
Let~$K$ be any algebra at all, and let~$\phi \colon \O_k(G) \to K$ be
any algebra homomorphism. Letting~$\delta_g$ denote the Dirac function
at~$g \in G$, we see that~$\phi$ is entirely determined by the
elements~$x_g = \phi(\delta_g) \in K$, which must be idempotents
summing to~$1$ and satisfying~$x_g x_h = 0$ whenever~$g \ne h$. If we
assume that the only idempotents in~$K$ are~$0$ and~$1$, then it
follows that there is one and only one~$g$ such that~$x_g = 1$ and all
other~$x_h$ are zero. Thus~$\phi(f) = f(g)1$.

Assume further that~$K$ is a Hopf algebra and that~$\phi$ is a Hopf
algebra homomorphism. Examination of the relation~$\Delta
(\phi(\delta_g)) = \phi \otimes \phi (\Delta (\delta_g))$ reveals
that~$g=1$. 

The key point is then the fact that this argument applies to~$K=
k[P]$, since the group algebra of a~$p$-group, in characteristic~$p$,
is indecomposable and thus has no other idempotents beside~$0$
and~$1$.
\end{proof}

The Lemma implies that~$\mathcal{ZP}(A\otimes B)$ is the trivial group
when~$A= \O_k(\z/2)$ and~$B= \O_k(\left( \z/2  \right)^r)$. Thus
what remains to be proved in Theorem~\ref{thm-main-elemab} follows
from Theorem~\ref{thm-main-zedmodtwo} by induction. 



\section{Lazy cohomology of function algebras} \label{sec-guillotkassel}

We now turn our attention to the result obtained by Kassel and the
author in~\cite{kassel}, and seek to adapt it to positive
characteristic. So now~$k$ is any field of characteristic~$p$, we
consider an arbitrary finite group~$G$, and we consider the second
lazy cohomology group~$H^2_\ell(\O(G))$ which was described at the end
of section~\ref{sec-defs}.

\subsection{Twists and~$R$-matrices}

Let~$F$ be a Drinfeld twist on the Hopf algebra~$\h$. If we put 
\[ R_F = F_{21} F^{-1}  \, ,  \]
then~$R_F \in \h \otimes \h$ is an~$R$-matrix. In~\cite{kassel}, we
have exploited the fact that, for~$\h = k[G]$ with~$G$ a finite group,
the~$R$-matrix~$R_F$ essentially determines~$F$ up to equivalence (a
more precise statement follows). What is more, a result of Radford
(\cite{radford}) shows that any~$R$-matrix at all for~$k[G]$ lives in fact in~$k[A]
\otimes k[A]$ where~$A$ is an abelian, normal subgroup of~$G$.

These results are valid regardless of the characteristic of~$k$, and
in order to extend the main theorem in~\cite{kassel} we are thus led
to investigate~$R$-matrices for Hopf algebras of the form~$k[A]$
where~$k$ has positive characteristic.

\subsection{$R$-matrices on abelian~$p$-group algebras}

We wish to prove the following result.

\begin{prop} \label{prop-R-matrix-char-p}
  Let~$k$ be a field of characteristic~$p$. If~$A$ is a finite
  abelian~$p$-group, then the only~$R$-matrix on the Hopf algebra~$k[A]$
  is the trivial one~$R= 1 \otimes 1$. More generally, if~$A$ is a
  finite abelian group, and if we write~$A = A_p \times A'$
  where~$A_p$ is the~$p$-Sylow subgroup of~$A$, then any~$R$-matrix
  on~$k[A]$ belongs to~$k[A'] \otimes k[A']$.
\end{prop}

\begin{proof}
Writing~$R= \sum_{a,b} \lambda_{ab} \, a \otimes b$,
where~$\lambda_{ab} \in k$ and~$a, b \in A$, we define a map $\phi_R \colon
\O(A) \longrightarrow  k[A]$ by the formula
\[ \phi_R(f) = \sum_{a,b} \lambda_{ab} \, f(a)b \, . \tag{*}  \]
The axioms for~$R$-matrices imply that~$\phi_R$ is a homomorphism of
Hopf algebras, as the reader will check. Thus
Lemma~\ref{lem-hopf-hom-O-kP} implies that~$\phi_R(f) = f(1)1$, for
all~$f \in \O(A)$. It follows that~$R= 1 \otimes 1$.

For the general statement, one establishes that (*) gives in fact a
bijection~$R \mapsto \phi_R$ between~$R$-matrices for~$k[A]$ and
homomorphisms of Hopf algebras~$\O(A) \to k[A]$; moreover this
bijection is natural in~$A$. Once this is granted, one starts with
an~$R$-matrix~$R$ for~$k[A]$ and composes~$\phi_R$ with the
projection~$k[A] \to k[A_p]$; this composition~$\O(A) \to k[A_p]$ must
the the trivial (augmentation) homomorphism, by
Lemma~\ref{lem-hopf-hom-O-kP}. It follows that~$\phi(f) \in k[A']$,
for all~$f$, so that precomposing with~$\O(A') \to \O(A)$ gives a
homomorphism of Hopf algebras~$\phi_{R'} \colon O(A') \to k[A']$
corresponding to an~$R$-matrix~$R'$ for~$k[A']$. By inspection, the
following diagram is commutative: 
\[ \begin{CD}
\O(A) @>{\phi_R}>> k[A] \\
@VVV                @AAA \\
\O(A') @>{\phi_{R'}}>> k[A']
\end{CD}
  \]
where the vertical maps are induced by the inclusion~$A' \to A$. It
follows that~$R$ is the image of~$R'$ under the map~$k[A'] \to k[A]$.
\end{proof}

This explains the relation~$F_{21}F^{-1} = 1 \otimes 1$ which we had
observed in example~\ref{ex-zedmodtwo}.

In order to complete the picture, at least when~$k$ is algebraically
closed, there remains only to state the following.

\begin{prop} \label{prop-char-prime-to-p}
  Let~$k$ be algebraically closed of characteristic~$p$, and let~$A$
  be a finite abelian group of order prime to~$p$. Then there is a
  bijection between the set of~$R$-matrices on~$k[A]$ and the bilinear
  forms on the Pontryagin dual of~$A$ with values in~$k^\times$.

Moreover, if~$R = \sum_i \lambda_i a_i \otimes b_i$ with~$\lambda_i
\in k$ and~$a_i, b_i \in A$, then the bilinear form corresponding
to~$R$ is alternating if and only if

\[ u_R := \sum_i \lambda_i a_i^{-1} \, b_i = 1  \, .  \]
\end{prop}

(In the proof we recall the relevant definitions. The element~$u_R$ is
called the Drinfeld element of~$R$.)

\begin{proof}
Let~$\widehat{A} = Hom(A, k^\times)$ be the Pontryagin dual
of~$A$. The discrete Fourier transform is the homomorphism 
\[ k[A] \longrightarrow \O(\widehat{A})  \]
defined by~$g \mapsto \hat g$, where~$\hat g (\chi ) = \chi (g)$
for~$\chi \in \widehat{A}$. The hypotheses on~$k$ garantee that the
discrete Fourier transform is an isomorphism of Hopf algebras.

As a consequence of this result, applied in fact to~$A \times A$, we
have a dictionnary between~$k[A] \otimes k[A]$ and~$\O(\widehat{A}\times
\widehat{A})$, that is the algebra of functions~$\widehat{A} \times
\widehat{A} \to k$. An~$R$-matrix for~$k[A]$ thus defines (and can be
defined by) a map
\[ b\colon \widehat{A} \times \widehat{A} \longrightarrow k^\times \, ,   \]
such that~$x \mapsto b(x, y)$ is a homomorphism for fixed~$y$, and~$y
\mapsto b(x, y)$ is a homomorphism for fixed~$x$. It is also immediate
that~$u_R = 1$ if and only if~$b(x^{-1}, x) = 1$ for all~$x\in
\widehat{A}$. This is the conclusion of the Proposition.
\end{proof}

It is instructive to see how this proof compares with the previous
one. The reader who is so inclined will check that, letting~$\a =
Spec(k[A])$ denote the affine group scheme associated to~$k[A]$, then
$R$-matrices on~$k[A]$ are in bijection with bilinear maps~$\a \times
\a \longrightarrow \g_m$. One can prove both
Proposition~\ref{prop-R-matrix-char-p} and
Proposition~\ref{prop-char-prime-to-p} using this language; in the
former case, the correspondence between~$R$ and~$\phi_R$ is
elucidated, while in the latter case the bilinear maps~$\a \times \a
\longrightarrow \g_m$ turn out to be equivalent to bilinear
maps~$\widehat{A} \times \widehat{A} \longrightarrow k^\times$ {\em
  via} the Fourier transform.

\subsection{The main theorem}

Let~$G$ be a finite group, and~$k$ an algebraically closed field of
characteristic~$p$. We let~$\b(G)$ denote the set of pairs~$(A, b)$
where~$A$ is an abelian, normal subgroup of~$G$ of order prime to~$p$,
and~$b$ is an alternating bilinear form~$\widehat{A} \times
\widehat{A} \to k^\times$ which is~$G$-invariant, and non-degenerate. 

Moreover, let~$\Int_k(G)$ denote the group of automorphisms of~$G$
induced by conjugation by elements of~$k[G]$, while~$\Inn(G)$ is the
group of inner automorphisms of~$G$; the quotient~$\Int_k(G)/\Inn(G)$
is a subgroup (which is often trivial in practice) of~$Out(G)$ .

\begin{thm} \label{thm-main-guillotkassel}
There is a map~$\Theta\colon H^2_\ell(G) \to \b(G)$ such that

(a) The subset~$\Theta^{-1}(1)$ is a subgroup of~$H^2_\ell(G)$
isomorphic to~$\Int_k(G) / \Inn(G)$;

(b) The fibres of~$\Theta $ are the left cosets of~$\Theta^{-1}(1)$;

(c) $\Theta $ is surjective if all the subgroups~$A$ involved in the
definition of~$\b(G)$ have odd order. In particular, $\Theta $ is
surjective if~$k$ has characteristic~$2$.

\end{thm}

\begin{proof}
  The proof of Theorem 4.5 in~\cite{kassel} goes through with only one
  simple change, emphasized below. The details of the following
  argument can all be found in {\em loc.\ cit.}

  To construct~$\Theta $, consider a twist~$F$ and the~$R$-matrix~$R_F
  = F_{21} F^{-1}$. There is a unique minimal, abelian, normal
  subgroup~$A$ of~$G$ such that~$R_F \in k[A] \otimes k[A]$, and by
  Proposition~\ref{prop-R-matrix-char-p}, {\em we know that the order
    of~$A$ is prime to~$p$}. By
  Proposition~\ref{prop-char-prime-to-p}, the~$R$-matrix~$R_F$ gives
  rise to a bilinear form~$b$ on~$\widehat{A}$. One can prove that the
  Drinfeld element of~$R_F$ is~$1$ so that~$b$ is alternating, and the
  minimality of~$A$ shows that~$b$ is non-degenerate; the fact
  that~$F$ is assumed to be~$G$-invariant shows that~$b$
  is~$G$-invariant. Thus it makes sense to put~$\Theta (F) = (A, b)$. 

The study of the fibres of the map~$\Theta $ so constructed is
identical to that carried out in~\cite{kassel}. Likewise for the
surjectivity of~$\Theta $ in good cases.
\end{proof}

\begin{ex}
Let~$G$ be a~$p$-group, and let~$k$ have
characteristic~$p$. Then~$\b(G)$ has only one element, by
construction, so we conclude from the Theorem that 
\[ H^2_\ell(\O(G)) = \Int_k(G)/\Inn(G)  \]
in this case. If moreover~$G$ is abelian, it follows
that~$H^2_\ell(\o(G)) = 0$, which we had observed with~$G=\z/2$
earlier. This example also shows that the condition that~$k$ be
algebraically closed cannot be removed.
\end{ex}


\appendix

\section{Cosimplicial groups obtained from cosimplicial sets}

In this Appendix we aim to prove Lemma~\ref{lem-goodwillie}. In
passing we recall the basic definitions of cosimplicial sets. The
material below grew out of an exchange on MathOverflow
which the author had with Tom Goodwillie and Fernando Muro.

Let~$\Delta $ be the simplex category, whose objects are~$\mathbf{0},
\mathbf{1}, \mathbf{2}, \ldots $ where~$\mathbf{n}$ is the ordered
set~$\{ 0, 1, 2, \ldots, n \}$, and whose morphisms are the
non-decreasing maps. A cosimplicial set is simply a functor
from~$\Delta $ to the category of sets. For the convenience of the
reader we recall that the morphisms in~$\Delta $ are compositions of
certain maps~$d^i$ and~$s^j$, satisfying
\[ d^j d^i = d^i d^{j-1} ~\textnormal{for}~ i < j  \]
\[ s^j d^i = \left\{ \begin{array}{l}
d^i s^{j-1} ~\textnormal{for}~ i < j \, , \\
Id ~\textnormal{for}~ i=j, j+1 \, , \\
d^{i-1} s^j ~\textnormal{for}~ i > j+1 \, , 
\end{array}\right.  \]
\[ s^j s^i = s^i s^{j+1} ~\textnormal{for}~ i \le j \, .   \]
Moreover these ``are enough''; that is, one can show that a
cosimplicial set~$X^*$ is precisely defined by a set~$X^n$ for each
integer~$n$ (we say that~$X^n$ is in ``degree~$n$''), together with
maps~$d^i = d^i_n \colon X^n \to X^{n+1}$ and~$s^j = s^j_n \colon X^n
\to X^{n-1}$ (with~$0 \le i \le n+1$ and~$0 \le j \le n-1$) satisfying
the relations above.

Given an integer~$m\ge 0$, there is a cosimplicial set which can be
called the ``free cosimplicial set on one point in degree~$m$'', and
which is given by~$Hom_\Delta (\mathbf{m}, \mathbf{n})$ in
degree~$n$. However, we will instead consider the {\em
  semi}-cosimplicial set~$F_m$ which in degree~$n$ consists of all
{\em injective} maps~$\mathbf{m} \to \mathbf{n}$ in~$\Delta $. Recall
that ``semi-cosimplicial'' means that that~$F_m$ is endowed with
cofaces, but no codegeneracies. Note also that~$F_m^n$ is empty for~$n
< m$. We shall also need to speak of the cosimplicial set which is
reduced to a point in every degree; we call it ``the cosimplicial
point''. 

The next Lemma says that any cosimplicial set is almost free as a
semi-cosimplicial set, except for the presence of cosimplicial points.

\begin{lem}[Goodwillie]
Any cosimplicial set is a disjoint union of cosimplicial points and
copies of~$F_m$ (for various values of~$m$), as semi-cosimplicial set.
\end{lem}

\begin{proof}
The {\em dual} of this statement is probably more familiar to the
reader. Namely in a simplicial set~$S_*$, if we call non-degenerate
the simplices which are not in the image of any degeneracy map, then
any element~$x \in S_*$ can be written uniquely~$x= s_{i_q} \cdots
s_{i_1} y$ with~$i_1 \le i_2 \le \cdots \le i_q$ and~$y$
non-degenerate. 

Dually, in a cosimplicial set~$X$, call an element a {\em root} of~$X$
if it is not in the image of any coface map. Then any~$x \in X^*$ can
(almost tautologically) be written~$ x= d^{i_1} \cdots d^{i_q} y$
where~$y$ is a root and~$i_1 \le i_2 \le \cdots \le i_q$; more
importantly, if~$y$ can be taken in degree~$> 0$, then this writing is
unique; as for roots in degree~$0$, they generate either a
cosimplicial point or a copy of~$F_0$.  We let the proof of this fact
as a (not entirely painless) exercise (start by proving that the
relation~$d^i(x) = d^j(y)$, when~$x$ and~$y$ are roots, implies
that~$x=y$ and either~$i=j$ or the degree of~$x$ is~$0$). We point out
however that the presence of codegeneracies is crucial here (for
example the relations~$s^i d^i = Id$ guarantee that the cofaces are
injective).

The lemma follows immediately from this. The various copies of~$F_m$
are indexed by the set of roots of~$X$; from now on the word ``root''
will exclude the elements of degree~$0$ which generate a cosimplicial
point.
\end{proof}

\begin{coro}
Let~$k$ be any ring, and let~$k[X]^*$ be the cosimplicial~$k$-module
obtained by taking in degree~$n$ the free~$k$-module
on~$X^n$. Then for~$n > 0$
\[ H^n(k[X]^*) = \bigoplus_{r} H^n(k[F_{m_r}]^*) \, .   \]
where~$r$ runs through the roots of~$X$, and~$m_r$ is the degree
of~$r$. 
\end{coro}

We have used that the cohomology of a cosimplicial point is~$0$ in
degrees~$> 0$. Let us now consider a specific cosimplicial set~$S^*$.

\begin{lem}[Muro]
Let~$S$ be any pointed set, and let~$S^n$ be the cartesian product
of~$n$ copies of~$S$. Define a cosimplicial set structure on~$S^*$ by 
\[ d^0(x_1, \ldots, x_n) = (*, x_1, \ldots, x_n) \, ,   \]
\[ d^i(x_1, \ldots, x_n) = (x_1, \ldots, x_i, x_i, \ldots, x_n) \, ,   \]
\[ d^n(x_1, \ldots, x_n) = (x_1, \ldots , x_n, *) \, ,   \]
while the codegeneracy~$s^i$ omits the~$i$-th entry. (Here~$*$ is the
base-point of~$S$). 

Then for any field~$k$ the cohomology~$H^n(k[S]^*)$ vanishes for~$n >
0$.
\end{lem}

\begin{proof} The trick is to consider the dual chain complex. Let~$V=
  k[S]^1$, so that~$k[S]^n= V^{\otimes n}$, and let~$R= Hom_k(V,
  k)$. Then~$R$ can be seen as the vector space of~$k$-valued
  functions on~$S$, and as such is a ring. We have~$Hom_k(V^{\otimes
    n}, k) = R^{\otimes n}$.

If we now inspect the chain complex~$Hom_k(k[S]^*, k)$, we recognize
the Hochschild complex of the ring~$R$ (with values in
the~$R$-module~$k$, the module structure being given by evaluation at
the base point of~$S$). Since~$R$ is a product of~$N$ copies of~$k$,
where~$N$ is the cardinal of~$S$, the K\"unneth formula shows then
that~$H_n( Hom_k(k[S]^*, k) )= 0$ for~$n > 0$. Therefore, we also have~$H^n(
k[S]^*) = 0$ for~$n > 0$.
\end{proof}

\begin{coro}
Let~$X^*$ be any cosimplicial set. Then the
group~$H^n(k[X]^*)$ vanishes for~$n > 0$.
\end{coro}

\begin{proof}
Since this holds for the example~$S^*$ of the Lemma, we gather from the
previous Corollary that~$H^n(k[F_m]^*) = 0$ for~$n > 0$ whenever~$m$
is one of those integers such that~$F_m$ shows up in the decomposition
of~$S^*$. However, whatever the integer~$m$, if suffices to take~$S$
with~$m+1$ elements~$x_0= *, x_1, \ldots, x_m$ to obtain a root~$(x_1,
\ldots, x_m)$ in degree~$m$ for the cosimplicial set~$S^*$.

We conclude that $H^n(k[F_m]^*) = 0$ for all~$m \ge 0$ and all~$n >
0$. Thus from the previous Corollary, $H^n(k[X]^*) = 0$ for any~$X^*$.
\end{proof}

\bibliography{myrefs}
\bibliographystyle{amsalpha}

\end{document}